\documentclass[11pt,reqno]{amsart}

 \newtheorem{thm}{Theorem}[section]
 
 \newtheorem{lem}[thm]{Lemma}
 \newtheorem{prop}[thm]{Proposition}
 \theoremstyle{definition}
 \newtheorem{defn}[thm]{Definition}
 \theoremstyle{remark}
 \newtheorem{rem}[thm]{Remark}
 \newtheorem*{ex}{Example}
 \numberwithin{equation}{section}

\usepackage{amssymb,amsthm,amsmath}

\begin{document}

\title[Lagrangian submanifolds] {Lagrangian submanifolds  in complex space forms  satisfying an improved equality involving $\delta(2,2)$}

\author[B.-Y. Chen]{Bang-Yen Chen }
\address{Department of Mathematics \\
	Michigan State University \\East Lansing, Michigan 48824--1027, U.S.A.}
\email{bychen@math.msu.edu}

\author[A. Prieto-Mart\'{i}n]{Alica Prieto-Mar\'{i}n }
\address{Department of Geometry and Topology \\University of Seville \\Apdo. de Correos 1160, 41080 - Sevilla, Spain} \email{aliciaprieto@us.es}

\author[X. Wang]{Xianfeng Wang}
\address{School of Mathematical Sciences and LPMC \\
Nankai University \\
Tianjin 300071,  P. R. China}
\email{wangxianfeng@nankai.edu.cn}



\thanks{A portion of this work was done while the second author was visiting Michigan State University in 2011, supported by a Fundaci\'{o}n C\'{a}mara grant, University of Sevilla, Spain. The third author was supported by the NSFC No.11171175 and the ``Fundamental Research Funds for the Central Universities''}

\date{}

\begin{abstract}  It was proved in \cite{CD,cdvv}  that every Lagrangian submanifold
$M$ of a complex space form $\tilde M^{5}(4c)$ of constant holomorphic sectional
curvature $4c$ satisfies the following optimal inequality:
\begin{align}\tag{A}\delta(2,2)\leq \text{\small$\frac{25}{4}$} H^{2}+8c,\end{align} where $H^{2}$ is the squared mean curvature and $\delta(2,2)$ is a $\delta$-invariant on $M$ introduced by the first author. This optimal inequality improves a special case of an earlier inequality obtained in [B.-Y. Chen, Japan. J. Math. {\bf 26} (2000), 105--127].

The main purpose of this paper is to classify Lagrangian submanifolds of $\tilde M^{5}(4c)$ satisfying the equality case of the  improved inequality (A).\end{abstract}

\maketitle

\def\<{\left < }
\def\>{\right >}
\def\({\left ( }
\def\){\right )}
\def\o{\omega }
\def\csch{\,{\rm csch}\,}
\def\sech{\,{\rm sech}\,}
\def\a{\alpha}\def\b{\beta }
\def\va{\varphi}
\def\ii{\hskip.006in {\rm i}\hskip.006in}
\def\g{\gamma}
\def\e{\eqref}
\def\tn{\tilde\nabla}

\section{Introduction} 

Let $\tilde M^n$ be a K\"ahler $n$-manifold with the complex structure $J$, a K\"ahler metric $g$ and the K\"ahler 2-form $\omega$.  An isometric immersion  $\psi:M\to \tilde M^n$ of a Riemannian $n$-manifold $M$ into  $\tilde M^n$ is called {\it Lagrangian\/} if $\psi^*\omega=0$. 

Let $\tilde M^n(4c)$ denote a K\"ahler $n$-manifold  with constant holomorphic sectional curvature $4c$, called a {\it complex space form}. A complete simply-connected complex space form $\tilde M^{n}(4c)$ is holomorphically isometric to the complex Euclidean $n$-plane ${\bf C}^n$, the complex projective $n$-space $CP^n(4c)$, or a complex hyperbolic $n$-space $CH^n(4c)$ according to $c=0,\, c>0$ or $c<0$, respectively.

 B.-Y. Chen introduced in 1990s new Riemannian invariants $\delta(n_1,\ldots,n_k)$. 
For any $n$-dimensional submanifold $M$ in a real space form $R^{m}(c)$ of constant curvature $c$, he 
 proved the following sharp general inequality (see \cite{c00a,book} for details):
\begin{equation}\begin{aligned}\label{1.1} \delta{(n_1,\ldots,n_k)} \leq &\, {{n^2(n+k-1-\sum n_j)}\over{2(n+k-\sum n_j)}} H^2 \\&+{1\over2} \Big({{n(n-1)}}-\sum_{j=1}^k {{n_j(n_j-1)}}\Big)c.\end{aligned}\end{equation}
 
For Lagrangian submanifolds in a complex space form $\tilde M^n(4c)$, we have
\vskip.05in

 {\bf Theorem A.} {\it  Let $M$ be an $n$-dimensional Lagrangian submanifold in a complex  space form $\tilde M^n(4c)$ of constant holomorphic sectional curvature $4c$. Then  inequality \e{1.1} holds  for each $k$-tuple  $(n_1,\ldots,n_k)\in\mathcal S(n)$.}
\vskip.05in

The following result from \cite{c00b} extends a result in \cite{cdvv2} on $\delta(2)$.
\vskip.05in

 {\bf Theorem B.} {\it  Every Lagrangian submanifold of a complex space form $\tilde M^n(4c)$ is minimal if it satisfies the equality case of \eqref{1.1} identically.}
\vskip.05in

Theorem B was improved recently in \cite{CD,cdvv} to the following inequality.

\vskip.05in

 {\bf Theorem C.} {\it  Let $M$ be an $n$-dimensional Lagrangian submanifold of $\tilde M^n(4c)$. Then, for an  $(n_1,\ldots,n_k)\in {\mathcal S}(n)$ with $\sum_{i=1}^{k} n_{i}<n$, we have
\begin{equation}\begin{aligned} \label{1.2}  \delta(n_1,\ldots,n_k) \leq &\, \text{\small$\dfrac{n^2 \Big\{\Big(n- {\sum_{i=1}^k}  n_i+ 3k-1\Big)-6\, {\sum_{i=1}^k} (2+ n_{i})^{-1}\Big\}}{2\Big\{\Big(n-{\sum_{i=1}^k}n_i+3k+ 2\Big)-6\,{\sum_{i=1}^k} (2+ n_{i})^{-1}\Big\}}$} H^2\
\\&\hskip.9in  +\text{$\frac{1}{2}$}\Big\{n(n-1)-\text{\small$\sum$}_{i=1}^k n_i(n_i-1)\Big\}c.\end{aligned}\end{equation}

The equality sign holds at a point $p\in M$ if and only if there is an orthonormal basis $\{e_1,\ldots,e_n\}$ at $p$ such that  the second fundamental form $h$ satisfies}
\begin{equation}\begin{aligned} \label{1.3}&h(e_{\alpha_i},e_{\beta_i})=\text{\small$\sum$}_{\gamma_i} h^{\gamma_i}_{\alpha_i \beta_i} Je_{\gamma_i}+\frac{3\delta_{\alpha_i\beta_i} }{2+n_i}\lambda Je_{N+1},\;\;  \text{\small$\sum$}_{\alpha_i=1}^{n_i} h^{\gamma_i}_{\alpha_i\alpha_i}=0,
\\&  h(e_{\alpha_i},e_{\alpha_j})=0,\; i\ne j; 
\; h(e_{\alpha_i},e_{N+1})=\frac{3\lambda}{2+n_i} J e_{\alpha_i},\;   h(e_{\alpha_i},e_u)=0,
\\& h(e_{N+1},e_{N+1})=3\lambda Je_{N+1},
\; h(e_{N+1},e_u)=\lambda Je_u , \; N=n_1+\cdots+n_k,
\\&h(e_u,e_v)=\lambda \delta_{uv} Je_{N+1},
\;  i,j=1,\ldots,k; \, u,v=N+2,\ldots,n.\end{aligned}\end{equation}
\vskip.05in

For simplicity, we call a Lagrangian submanifold of a complex space form  {\it $\delta(n_{1},\ldots,n_{k})$-ideal} (resp., {\it  improved $\delta(n_{1},\ldots,n_{k})$-ideal}\/) if it satisfies the equality case of \e{1.1} (resp., the equality case of \e{1.2}) identically.

For $k=2$ and $n_{1}=n_{2}=2$,  Theorem C reduces to the following.
\vskip.04in 

{\bf Theorem D.} {\it Let $M$ be a Lagrangian submanifold in a complex  space form $\tilde M^5(4c)$ of constant holomorphic sectional curvature $4c$. Then we have
\begin{equation} \label{1.4} \delta(2,2)\leq \text{\small$\frac{25}{4}$} H^{2}+8c.\end{equation}

If the equality sign of \e{1.4} holds identically, then with respect some suitable orthonormal frame $\{e_{1},\ldots,e_{5}\}$ the second fundamental form $h$ satisfies
\begin{equation}\begin{aligned}
\label{1.5} &h(e_{1},e_{1})=\alpha Je_{1}+\beta Je_{2}+\mu Je_{5},\; h(e_{1},e_{2})=\beta Je_{1}-\alpha Je_{2},\;\;
\\& h(e_{2},e_{2})=-\alpha Je_{1}-\beta Je_{2}+\mu Je_{5},\;\;
\\ &h(e_{3},e_{3})=\gamma Je_{3}+\delta Je_{4}+
\mu Je_{5}, \; h(e_{3},e_{4})=\delta Je_{3}-\gamma Je_{4},\;\;\\& h(e_{4},e_{4})=-\gamma Je_{3}-\delta Je_{4}+\mu Je_{5},\; h(e_{5},e_{5})=4\mu Je_{5},
\\& h(e_{i},e_{5})=\mu Je_{i},\; i\in \Delta; \; 
  h(e_{i},e_{j})=0,\;\; otherwise,
\end{aligned}\end{equation}
for some functions $\alpha,\beta,\gamma, \delta,\mu$, where $\Delta=\{1,2,3,4\}$.}
\vskip.05in

The classification of $\delta(2,2)$-ideal Lagrangian submanifolds in complex space forms $\tilde M^{5}(4c)$ is done in \cite{CP}.
In this paper we classify  improved $\delta(2,2)$-ideal Lagrangian submanifolds in $\tilde M^{5}(4c)$. The main results of this paper are stated as Theorem \ref{T:6.1}, Theorem   \ref{T:7.1} and  Theorem \ref{T:8.1}.

\section{Preliminaries}
\subsection{Basic formulas}
Let $\tilde M^n(4c)$ denote a complete simply-connected K\"ahler $n$-manifold  with constant holomorphic sectional curvature $4c$. Then $\tilde M^n(4c)$ is holomorphically isometric to the complex Euclidean $n$-plane ${\bf C}^n$, the complex projective $n$-space $CP^n(4c)$, or a complex hyperbolic $n$-space $CH^n(-4c)$ according to $c=0,c>0$ or $c<0$.

Let $M$ be a Lagrangian submanifold of  $\tilde M^n(4c)$.
We denote the Levi-Civita connections of $M$ and  $\tilde M^n(4c)$ by $\nabla$ and $\tilde \nabla$, respectively.
The formulas of Gauss and Weingarten are given respectively by (cf. \cite{book})
\begin{align}\label{2.1} &\tn_X Y = \nabla_X Y + h(X,Y),\;\; \tn_X \xi = -A_\xi X + D_X \xi,\end{align}
for tangent vector fields $X$ and $Y$ and normal vector fields  $\xi$, where $h$ is the second fundamental form, $A$ is the shape operator and $D$ is the normal connection. 

The second fundamental form  and the shape operator are related by
$$\<h(X,Y),\xi\> = \<A_\xi X,Y\>. $$
The mean curvature vector $\overrightarrow{H}$ of $M$ is defined by $\overrightarrow{H}={1\over n}\,\hbox{trace}\,h$ and the {\it squared mean curvature} is given by $H^{2}=\<\right.\hskip-.02in \overrightarrow{H},\overrightarrow{H}\hskip-.02in \left.\>.$

 For Lagrangian submanifolds, we have (cf. \cite{book,CO})
\begin{align}\label{2.3} &D_X JY = J \nabla_X Y,
\\\label{2.4} &A_{JX} Y = -J h(X,Y)=A_{JY}X.\end{align}
Formula \e{2.4} implies that $\<h(X,Y),JZ\>$ is totally symmetric.  

The equations of Gauss and Codazzi are given respectively by
\begin{align} \label{Gauss} & \<R(X,Y)Z,W\> =  \<A_{h(Y,Z)} X,W\> - \<A_{h(X,Z)}Y,W\>\\&
\notag \hskip1.3in +c(\<X,W\>\<Y,Z\>-\<X,Z\>\<Y,W\>),
\\ &\label{Codazzi} (\nabla_{X} h)(Y,Z) = (\nabla_{Y} h)(X,Z),\end{align}
where $R$ is the curvature tensor of $M$ and $\nabla h$ is defined
by \begin{align}\label{2.7}(\nabla_{X} h)(Y,Z) = D_X h(Y,Z) - h(\nabla_X Y,Z) - h(Y,\nabla_X Z).\end{align}

For an orthonormal basis $\{e_1,\ldots,e_n\}$ of $T_pM$, we put $$h^i_{jk}=\<h(e_j,e_k),Je_i\>,\;\;i,j,k=1,\ldots,n.$$
It follows from \eqref{2.4} that $ h^i_{jk}=h^j_{ik}=h^k_{ij}.$

\subsection{$\delta$-invariants}

Let $M$ be a  Riemannian $n$-manifold. Denote by $K(\pi)$ the sectional curvature of  a plane section $\pi\subset T_pM$, $p\in M$. For any orthonormal basis
$e_1,\ldots,e_n$ of $T_pM$, the scalar curvature $\tau$ at $p$ is
 $\tau(p)=\sum_{i<j} K(e_i\wedge e_j).$

Let $L$ be a $r$-subspace of $T_pM$ with $r\geq 2$  and $\{e_1,\ldots,e_r\}$ an orthonormal basis of $L$. The scalar curvature $\tau(L)$ of  $L$ is defined by
\begin{align}\label{2.8}\tau(L)=\sum_{\alpha<\beta} K(e_\alpha\wedge e_\beta),\quad 1\leq \alpha,\beta\leq r.\end{align}

For given integers $n\geq 3$,  $k\geq 1$, we denote by $\mathcal S(n,k)$ the finite set  consisting of $k$-tuples $(n_1,\ldots,n_k)$ of integers  satisfying  
$2\leq n_1,\cdots,
n_k<n$  and $\sum_{{j=1}}^{k} i\leq n.$ 

Put ${\mathcal S}(n)=\cup_{k\geq 1}\mathcal S(n,k)$.
For each $k$-tuple $(n_1,\ldots,n_k)\in \mathcal S(n)$, the first author introduced in 1990s the Riemannian invariant $\delta{(n_1,\ldots,n_k)}$ by
\begin{align}\label{2.9} \delta(n_1,\ldots,n_k)(p)=\tau(p)- \inf\{\tau(L_1)+\cdots+\tau(L_k)\},\;\; p\in M,\end{align} where $L_1,\ldots,L_k$ run over all $k$ mutually orthogonal subspaces of $T_pM$ such that $\dim L_j=n_j,\, j=1,\ldots,k$ (cf. \cite{book} for details).

\subsection{Horizontal lift of Lagrangian submanifolds}

The following link between Legendrian submanifolds and Lagrangian submanifolds is due to  \cite{R} (see also \cite[pp. 247--248]{book}).
\vskip.04in

\noindent {\it Case} (i): $CP^n(4)$.  Consider Hopf's fibration $\pi :S^{2n+1}\to CP^n(4).$  For a given point $u\in S^{2n+1}(1)$, the horizontal space at $u$ is the orthogonal complement of $\i u, \, \i=\sqrt{-1},$ with respect to the metric  on $S^{2n+1}$ induced from the metric on ${\bf C}^{n+1}$.   
Let $\iota : N \to CP^n(4)$ be a Lagrangian isometric immersion. Then there is a covering map
$\tau: \hat N \to N$ and a horizontal  immersion $\hat \iota :\hat N \to S^{2n+1}$ such that
$\iota\circ \tau=\pi \circ \hat \iota$.  Thus each Lagrangian immersion can be lifted locally (or globally if $N$ is simply-connected) to a Legendrian immersion of the same Riemannian manifold. In particular, a minimal Lagrangian submanifold of $CP^{n}(4)$ is lifted to a minimal Legendrian submanifold of the Sasakian $S^{2n+1}(1)$.

Conversely, suppose that $f: \hat N \to S^{2n+1}$ is a Legendrian isometric immersion. Then $\iota =\pi\circ f:  N\to  CP^n(4)$ is again a Lagrangian isometric immersion.  Under this correspondence the second fundamental forms $h^f$ and $h^\iota$ of $f$ and $\iota$ satisfy $\pi_*h^f=h^\iota$.  Moreover, $h^f$ is horizontal with respect to $\pi$.  
\vskip.04in

\noindent {\it Case} (ii): $CH^n(-4)$.  We consider  the complex number space  ${\bf C}^{n+1}_1$ equipped
with the pseudo-Euclidean metric:
$g_0=-dz_1d\bar z_1 +\sum_{j=2}^{n+1}dz_jd\bar z_j .$

 Consider 
$H^{2n+1}_{1}(-1)=\{z\in {\bf C}^{2n+1}_1: \<z,z\>=-1\}$ 
with the canonical Sasakian structure, where $\<\;\,,\;\>$ is the induced inner product.

Put $T'_z=\{ u\in {\bf C}^{n+1}:\<u,z\>=0\},\;\; H_1^1=\{\lambda\in {\bf  C}: \lambda\bar\lambda=1\}.$
 Then there is an $H^1_1$-action on $H_1^{2n+1}(-1)$,  $z\mapsto \lambda z$ and at each point $z\in H^{2n+1}_1(-1)$, the vector $\xi=-\i z$ is tangent to the flow of the action. Since the metric $g_0$ is Hermitian, we have $\<\xi,\xi\>=-1$.  The quotient space $H^{2n+1}_1(-1)/\sim$, under the identification induced from the action, is the complex
hyperbolic space $CH^n(-4)$ with constant holomorphic sectional curvature $-4$ whose complex structure $J$ is  induced from the complex structure $J$ on ${\bf C}^{n+1}_1$ via Hopf's fibration $\pi :H^{2n+1}_1(-1 )\to  CH^n(4c).$

 Just like case (i), suppose that $\iota: N \to CH^n(-4)$ is a Lagrangian immersion, then there is an isometric covering map
$\tau: \hat N \to N$ and a Legendrian immersion $f: \hat N \to H_1^{2n+1}(-1)$ such that $\iota\circ \tau=\pi\circ f$.  Thus every Lagrangian immersion into $CH^n(-4)$ an be lifted locally (or globally if $N$ is simply-connected) to a Legendrian immersion into $H^{2n+1}_1(-1)$. In particular,  Lagrangian minimal submanifolds of $CH^{n}(-4)$ are lifted to Legendrian minimal submanifolds of $H^{2n+1}_{1}(-1)$.
Conversely, if $f:\hat N \to H_1^{2n+1}(-1)$ is a Legendrian immersion, then $\iota =\pi\circ f:  N\to CH^n(-4)$ is a Lagrangian immersion.  Under this correspondence the second fundamental forms $h^f$ and $h^\iota$ are related by $\pi_*h^f=h^\iota$. Also, $h^f$ is horizontal with respect to $\pi$. 

Let $h$ be the second
fundamental form of $M$ in $S^{2n+1}(1)$ (or in $H^{2n+1}_1(-1)$). Since $S^{2n+1}(1)$ and $H^{2n+1}_1(-1)$ are totally umbilical  with one as its mean curvature in  ${\bf C}^{n+1}$ and in ${\bf C}^{n+1}_1$,
respectively, we have  
\begin{align}\label{2.10} \hat \nabla_XY=\nabla_XY+h(X,Y)-\varepsilon L,\end{align} 
where $\varepsilon=1$ if the ambient space is ${\bf C}^{n+1}$; and $\varepsilon=-1$ if it is ${\bf C}^{n+1}_1$.

\section{$H$-umbilical Lagrangian submanifolds and complex extensors}

 \subsection{$H$-umbilical Lagrangian submanifolds}
 \begin{defn}  A non-totally geodesic Lagrangian submanifold of a K\"ahler $n$-manifold is called {\it $H$-umbilical}  if its second fundamental form satisfies 
\begin{equation} \begin{aligned} \label{3.1}&
 h(e_j, e_j)\hskip-.01in =\hskip-.01in \mu
J e_n, \;\; h(e_j, e_n)\hskip-.01in =\hskip-.01in \mu J e_j,\; \; j=1,\ldots, n-1,
\\& h( e_n, e_n)= \varphi J e_n,\; h( e_j,e_k)= 0,\;\;  1\leq j\ne k\leq n-1, \end{aligned}\end{equation}  for some  functions $\mu,\varphi$ with respect to an orthonormal frame $\{ e_1,\ldots, e_n\}$.
If the ratio of $\varphi:\mu$ is a constant $r$,   the $H$-umbilical submanifold is said to be {\rm of ratio} $r$.  \end{defn}

If $G:N^{n-1}\rightarrow \mathbb E^n$ is a hypersurface of a Euclidean $n$-space $\mathbb E^n$ and $\gamma:I\rightarrow {\bf C}^*$ is a unit speed curve in  ${\bf C}^*={\bf C}-\{0\}$, then we may extend $G:N^{n-1}\rightarrow \mathbb E^n$ to an immersion  $I\times N^{n-1}\to {\bf C}^n$ by
$\gamma \otimes G:I\times N^{n-1}\rightarrow {\bf C}\otimes \mathbb E^n={\bf C}^n$,
where $(\gamma \otimes G)(s,p)=F(s)\otimes G(p)$ for $ s\in I, \; p\in N^{n-1}.$ This extension  of $G$ via tensor product $\otimes$ is called the {\it complex extensor\/} of $G$ via the {\it generating curve} $\gamma$.

$H$-umbilical Lagrangian submanifolds in complex space forms were classified in a series of papers by the first author (cf. \cite{tohoku,israel,tohoku2}). In particular, 
the following two results were proved in \cite{tohoku}.
\vskip.05in

 {\bf Theorem E.}\label{T:A} {\it Let $\iota:S^{n-1}\subset \mathbb E^n$ be  the unit hypersphere in $\mathbb E^n$ centered at the origin. Then every complex extensor of $\iota$ via a unit speed curve $\gamma:I\to \bf C^*$ is an $H$-umbilical  Lagrangian submanifold of ${\bf C}^n$ unless $\gamma$ is contained in a line through the origin (which gives a totally geodesic Lagrangian submanifold).} 
\vskip.05in

 {\bf Theorem F.} \label{T:B} {\it Let  $M$  be an  $H$-umbilical Lagrangian submanifold of $\,{\bf C}^n$ with $n\geq 3$. Then  $M$ is either a flat space or congruent to an open part of a complex extensor of $\iota:S^{n-1}\subset \mathbb E^n$ via a curve $\gamma:I\to \bf C^*$.}
\vskip.05in

\subsection{Legendre curves} 

A unit speed curve $z:I\to S^{3}(1)\subset {\bf C}^{2}$ (resp., $z:I\to H_{1}^{3}(-1)\subset {\bf C}_{1}^{2}$) is called  {\it Legendre} if $\<z',\ii z\>=0$. It was proved in \cite{israel} that a unit speed curve $z$ in $S^{3}(1)$ (resp., in $H^{3}_{1}(-1)$) is Legendre if and only if it satisfies
\begin{align}\label{3.2} z''=\ii \lambda z'- z \;\; ({\rm resp}., \; z''=\ii \lambda z'+ z)\end{align}
for a real-valued function $\lambda$. It is known in \cite{israel} that $\lambda$ is  the curvature function of  $z$ in $S^{3}(1)$ (resp., in $H^{3}_{1}(-1)$) (see also  \cite[Lemmas 3.1 and 3.2]{CC}).

\subsection{$H$-umbilical submanifolds with arbitrary ratio} 

We provide a general method to construct $H$-umbilical  Lagrangian submanifolds with any given ratio in $CP^{n}(4)$ via  curves in $S^{2}(\tfrac{1}{2})$ (resp., in  $CH^{n}(-4)$ via curves in  $H^{2}(-\tfrac{1}{2})$).

\begin{prop}  For any  real number $r$ there exist $H$-umbilical Lagrangian submanifolds of ratio $r$ in $CP^{n}(4)$ and in $CH^{n}(-4)$.
\end{prop}
\begin{proof} If $r=2$ this was done in \cite[Theorems 5.1 and 6.1]{israel}. 
If $r\ne 2$, $H$-umbilical  Lagrangian submanifolds of ratio $r$ can be constructed as follows:

\vskip.05in
 {\it Case} (a): $CP^{n}(4)$.
Let $S^2(\tfrac{1}{2})=\{ {\bf x}\in {\mathbb E}^3 ; \<{\bf x},{\bf x}\>=\tfrac{1}{4}\}$.
 The Hopf fibration $\pi$ from $S^3(1)$ onto $S^2(\tfrac{1}{2})\equiv CP^1(4)$ is given by (cf. \cite{CC})
\begin{align}\pi(z_{1},z_{2})=  \(z_{1}\bar z_{2},\text{$\tfrac{1}{2}$}( |z_{1}|^2-|z_{2}|^2) \),\;\; (z_{1},z_{2})\in S^3(1)\subset {\bf C}^2.  
\end{align}
For a Legendre curve $z$ in $S^3(1)$, the projection $\gamma_{z}=\pi\circ z$ is a curve in $S^2(\tfrac{1}{2})$. Conversely,   each curve $\gamma$ in $S^2(\tfrac{1}{2})$ gives rise to a horizontal lift $\tilde{\gamma}$ in $S^3(1)$ via $\pi$ which is unique up to a factor $e^{i\theta},\theta\in {\bf R}$. Notice that each horizontal lift of $\gamma$  is a Legendre curve in $S^3(1)$. Moreover,
since the Hopf fibration is a Riemannian submersion, each unit speed Legendre curve $z$ in $S^3(1)$  is projected to  a unit speed curve $\gamma_{z}$ in $S^2(\tfrac{1}{2})$ with the same curvature.  

It was known in \cite[Lemma 7.2]{israel} that,  for a given $H$-umbilical Lagrangian submanifold of ratio $r\ne 2$ in ${\tilde M}^n(4c)$, the function $\mu$ in \e{3.1} satisfies 
\begin{align} \label{3.4} \mu \mu''-\(\text{\small$\frac{r-3}{r-2}$}\)\mu'{}^2+(r-2)\mu^2((r-1)\mu^2+c)=0.
\end{align}

If $\mu$ is a non-trivial solution of \e{3.4} with $c=1$, then there is a unit speed curve $\gamma$ in $S^{2}(\tfrac{1}{2})$ whose curvature equals to  $r\mu$. Let $z$ be a horizontal lift of  $\gamma$ in $ S^3(1)$. Then $z$ is a unit speed Legendre curve satisfying $z''(x)= \ii r\mu z'(x)-z(x)$ (cf. \cite[Theorem 4.1]{israel} or \cite[Lemma 3.1]{CC}).

Consider the map $\psi:M^{5}\to S^{11}(1)\subset {\bf C}^{6}$ defined by
 \begin{align}\label{3.5} \psi(x,y_{1},\ldots,y_{5})=(z_{1}(x),z_{2}(x)y_{1},\ldots,\ldots,z_{2}(x)y_{5}),\;\; \sum_{j=1}^{5}y_{j}^{2}=1.\end{align} 
It follows from \cite[Theorem 4.1 and Lemma 7.2]{israel} that  $\pi\circ \psi$ is a $H$-umbilical Lagrangian submanifold of ratio $r$ in $CP^{n}(4)$ such that
\begin{equation} \begin{aligned} \label{8.5}&
 h(e_j,e_j)\hskip-.01in =\hskip-.01in \mu
J e_5, \;\; h(e_j, e_n)= \mu J e_j,\; \;
\\& h(e_n, e_n)= r\mu J e_n,\; h(e_j,e_k)= 0,\;\;  1\leq j\ne k\leq n-1, \end{aligned}\end{equation}  
with respect to suitable orthonormal frame $\{e_{1},\ldots,e_{5}\}$.

\vskip.05in
{\it Case} (b): $CH^{n}(-4)$. For a non-trivial solution of \e{3.4} with $c=-1$, we can construct an $H$-umbilical Lagrangian submanifold of $CH^{n}(-4)$ via the Hopf fibration $\pi: H^{3}_{1}(-1)\to CH^{1}(-4)\equiv H^{2}(-\tfrac{1}{2})$ in a similar way as case (a), where 
 \begin{align}\label{3.76}
\pi(z_{1},z_{2})=  \(z_{1}\bar z_{2},\text{$\tfrac{1}{2}$}( |z_{1}|^2+|z_{2}|^2) \),\;\; (z_{1},z_{2})\in H_{1}^3(-1)\subset {\bf C}_{1}^2,   \end{align}
and $H^{2}(-\tfrac{1}{2})=\{ (x_{1},x_{2},x_{3})\in {\mathbb E}^3_{1} \,:\,
x_1^2-x_2^2-x_3^2=\tfrac{1}{4}, \, x_1 \geq \tfrac{1}{2}  \}$ 
 is the model of the real projective plane of curvature $-4$.
\end{proof}

\subsection{Classification of $H$-umbilical submanifolds of ratio 4}

The equation of Gauss  and \e{3.1} imply that $H$-umbilical Lagrangian submanifolds of ratio $r\ne 4$ in complex space forms  contain no open subsets of constant sectional curvature. Hence we conclude from \cite[Theorems 4.1 and 7.1]{israel} and \S3.3 the following results. 

\begin{lem} \label{L:3.4} An  $\,H$-umbilical  Lagrangian submanifold $M$ of ratio 4 in $CP^{5}(4)$ is congruent to an open portion of $\pi\circ \psi$, where $\pi:S^{11}(1)\to CP^{5}(4)$ is Hopf's fibration,  $\psi:M\to S^{11}(1)\subset {\bf C}^{6}$ is given by
\begin{align}\label{3.8} \psi(t,y_{1},\ldots,y_{5})=(z_{1}(t),z_{2}(t){\bf y}),\;\; \{{\bf y}\in {\mathbb E}^{5}:\<{\bf y},{\bf y}\>=1\},\end{align}
 and $z=(z_{1},z_{2}):I\to S^{3}(1)\subset {\bf C}^{2}$ is a unit speed Legendre curve satisfying $z''= 4\ii \mu z'-z,$ and $\mu$ is a nonzero solution of 
$2\mu \mu''-\mu'{}^2+4\mu^2(3\mu^2+1)=0$.
\end{lem}

Let $M$ be an $H$-umbilical  Lagrangian submanifold  in $CH^{5}(-4)$ satisfying \e{3.1}. We  may assume that $\mu$  is defined on an open interval $I\ni 0$.
Since $H$-umbilical  submanifolds of ratio 4 in $CH^{5}(-4)$ contain no open subsets of constant  curvature, Theorems 4.2 and 9.1 of  \cite{israel} and results in \S3.3 imply the following classification of $H$-umbilical  submanifolds of ratio 4 in $CH^{5}(-4)$.

\begin{lem}\label{L:3.5} An  $\,H$-umbilical Lagrangian submanifold $M$ of ratio 4 in $CH^{5}(-4)$ is congruent to an open part of $\pi\circ \psi$, where  $\pi: H^{11}_{1}(-1)\to CH^{5}(-4)$ is Hopf's fibration and $\psi:M\to H_{1}^{11}(-1)\subset {\bf C}_{1}^{6}$ is either one of 
\begin{align} &\psi(t,y_{1},\ldots,
y_{4})=(z_{1}(t),z_{2}(t){\bf y}), \;\; \{{\bf y}\in {\mathbb E}^{5}:\<{\bf y},{\bf y}\>=1\},
\\& \psi(t,y_{1},\ldots,y_{4})=(z_{1}(t){\bf y},z_{2}(t)),\;\;\{{\bf y}\in {\mathbb E}_{1}^{5}:\<{\bf y},{\bf y}\>=-1\},\end{align}
 where $z$ is a unit speed Legendre curve in $H_{1}^{3}(-1)$ satisfying $z''= 4\ii \mu z'+z$ and $\mu$ is a non-trivial solution of 
 $2\mu \mu''-\mu'{}^2+4\mu^2(3\mu^2-1)=0$; or $\psi$ is 
\begin{equation} \begin{aligned} \label{3.12}&\psi(t,u_{1},\ldots,u_{4})=\sqrt{\mu}e^{\ii \int_0^t \mu(t) dt}\Bigg(1+{1\over 2}\sum_{j=1}^4 u_j^2 -\ii t+ \frac{1}{2\mu} - \frac{1}{2\mu(0)}, \\ & \hskip.5in \(i\mu(0)-\frac{\mu'(0)}{2\mu(0)}\)\Bigg({1\over 2}\sum_{j=1}^4 u_j^2 -\ii t+\frac{1}{2\mu} - \frac{1}{2\mu(0)}\Bigg),u_1,\ldots, u_4\Bigg), \end{aligned}\end{equation}  
where  $z=(z_{1},z_{2}):I\to H_{1}^{3}(-1)\subset {\bf C}^{2}_{1}$ is a unit speed Legendre curve and 
 $\mu$ is a non-trivial solution of $\mu'^{2}={4\mu^{2}}(1-\mu^{2})$.
\end{lem}

\begin{ex} {\rm  It is easy to verify that $\mu=\sech 2t$ is a non-trivial solution of $\mu'^{2}={4\mu^{2}}(1-\mu^{2})$. Using $\mu=\sech 2t$,  \e{3.12} reduces to
\begin{equation} \begin{aligned} \label{3.13}& \psi(t,u_{1},\ldots,u_{4})= \text{\small$ \frac{e^{{\rm i} \tan^{-1}(\tanh t)}}{\sqrt{\cosh 2t}}\Bigg(\frac{1}{2}$}-\ii t+ \text{\small$ \frac{1}{2}\sum_{j=1}^4 $}u_j^2+\text{\small$\frac{\cosh 2t}{2}$},
\\&\hskip1.7in  t- \text{\small$ \frac{\ii}{2}$} +\text{\small$ \frac{\ii}{2}\sum_{j=1}^4$} u_j^2+ \text{\small$\frac{\ii \cosh 2t}{2}$},u_1,\ldots,u_4\text{\small$\Bigg)$}.
 \end{aligned}\end{equation}  
It is direct to verify that \e{3.13} satisfies $\<\psi,\psi\>=-1$ and the composition $\pi\circ \psi$ gives rise to an $H$-umbilical Lagrangian submanifold of ratio 4 in $CH^{5}(-4)$.
}\end{ex}

\section{Some Lemmas}

We need the following lemmas for the proof of the main theorems.

\begin{lem}\label{L:4.1}  Let $M$ be an improved $\delta(2,2)$-ideal Lagrangian submanifold of $\tilde M^5(4c)$. Then  with respect to some orthonormal frame $\{e_{1},\ldots,e_{5}\}$ we have
 \begin{equation}\begin{aligned}\label{4.1} &h(e_{1},e_{1})=a Je_{1}+\mu Je_{5},\, h(e_{1},e_{2})=-a Je_{2},\,\\& h(e_{2},e_{2})=-a Je_{1}+\mu Je_{5},\; h(e_{3},e_{3})=b Je_{3}+\mu Je_{5},\; \\& h(e_{3},e_{4})=-b Je_{4},\; 
h(e_{4},e_{4})=-b Je_{3}+\mu Je_{5},\;\;
\\& h(e_{i},e_{5})= \mu Je_{i},\; i\in \Delta,  \; h(e_{5},e_{5})=4\mu Je_{5},
\; \\& h(e_{i},e_{j})=0,\;\; otherwise.
\end{aligned}\end{equation}
\end{lem}
\begin{proof}  Under the hypothesis, we have \e{1.5} with respect to an orthonormal frame $\{e_{1},\ldots,e_{5}\}$.  Thus, after applying \cite[Lemma 1]{c00b} to $V={\rm Span}\,\{e_{1},e_{2}\}$ and $V={\rm Span}\,\{e_{3},e_{4}\}$, we obtain \e{4.1}. 
\end{proof} 

Let us put \begin{align}\label{4.2}\nabla_{X}e_{i}=\sum_{j=1}^{5}\o_{i}^{j}(X)e_{j},\;\; i=1,\ldots,5,\;\; X\in TM^{5}.\end{align}
Then $\o_{i}^{j}=-\o_{j}^{i},\, i,j=1,\ldots,5$. 
 
If $\mu=0$, then  $M$ is a minimal Lagrangian submanifold according \e{4.1}.  Such submanifolds in complex space forms $\tilde M^{5}(4c)$ have been classified in \cite{CP}.

 If $a=b=0$ and $\mu\ne 0$, then $M$ is an $H$-umbilical Lagrangian submanifold with ratio 4. Therefore,  from now on we assume that  $a,\mu\ne 0$.

\begin{lem}\label{L:4.2}  Let $M$ be a Lagrangian submanifold of $\tilde M^5(4c)$ whose second fundamental form satisfies  \e{4.1} with $a,b,\mu \ne 0$. Then we have 
\begin{equation}\begin{aligned}\label{4.3}  &\nabla_{e_{1}}e_{1}=\frac{e_{2}a}{3a} e_{2}-\nu e_{5},\; \nabla_{e_{1}}e_{2}=-\frac{e_{2}a}{3a}  e_{1},\; \nabla_{e_{2}}e_{1}=-\frac{e_{1}a}{3a} e_{2}, 
\\&   \nabla_{e_{2}}e_{2}=\frac{e_{1}a}{3a} e_{1}-\nu e_{5},
\; \nabla_{e_{3}}e_{3}=\frac{e_{4}b}{3b} e_{4}-\nu e_{5},\; \nabla_{e_{3}}e_{4}=-\frac{e_{4}b}{3b}  e_{3},
\\& \nabla_{e_{4}}e_{3}=-\frac{e_{3}b}{3b} e_{4},\; \; \nabla_{e_{4}}e_{4}=\frac{e_{3}b}{3b} e_{3}- \nu e_{5}, \;  \nabla_{e_{i}}e_{5}=\nu e_{i},\; i\in \Delta, \; \\&\nabla_{e_{k}}e_{j}=0,\;\; otherwise,\end{aligned}\end{equation}
 with $\nu=\frac{1}{2}e_{5}(\ln \mu)=-e_{5}(\ln a)=-e_{5}(\ln b)$, where $\Delta=\{1,2,3,4\}$.
Moreover, we have 
\begin{align}\label{4.4}  & e_{j}\mu=0,  j\in \Delta,
\;\;   e_{1}b=e_{2}b =e_{3}a=e_{4}a=0.\end{align}
\end{lem}
\begin{proof} This lemma is obtained from Codazzi's equations via Lemma \ref{4.1} and \e{4.2} and long computations.
\end{proof}

\begin{lem}\label{L:4.3} Under the hypothesis of Lemma \ref{L:4.2}, we have

\vskip .04in
\begin{itemize}
\item[{\rm (a)}] $T_{0}$ is a totally geodesic distribution, i.e. $T_{0}$ is integrable whose leaves are totally geodesic submanifolds;

\item[{\rm (b)}]  $T_{0}\oplus T_{1}$ and $T_{0}\oplus T_{2}$ are totally geodesic distributions;

\item[{\rm (c)}]  $T_{1}$ and $T_{2}$ are spherical  distributions, i.e. $T_{1}$, $T_{2}$  are integrable distributions whose leaves are totally umbilical submanifolds with parallel mean curvature vector, 
\end{itemize}
 where
$T_{0}={\rm Span}\{e_{5}\}, T_{1}={\rm Span}\{e_{1},e_{2}\}$ and $T_{2}={\rm Span}\{e_{3},e_{4}\}.$
\end{lem}
\begin{proof} Since the distribution $T_{0}$ is of rank one, it is integrable. Moreover, since $\nabla_{e_{5}}e_{5}=0$ by Lemma \ref{L:4.2}, the integral curves of $e_{5}$ are geodesics in $M$. Thus we have statement (a).
Statement (b) follows easily from \e{4.3}.

To prove statement (c), first we observe that $[e_{1},e_{2}]\in T_{1}$ and $[e_{3},e_{4}]\in T_{2}$ follow from \e{4.3}. Thus $T_{1},T_{2}$ are integrable. Also, it follows from \e{4.3} that the second fundamental form $h_{1}$ of a leaf $\mathcal L_{1}$ of $T_{1}$ in $M$ is given by
\begin{align} h_{1}(X,Y)=-\nu g_{1}(X_{1},Y_{1})e_{5},\;\; X_{1},Y_{1}\in T\mathcal L_{1},\end{align}
 where $g_{1}$ is the metric of $\mathcal L_{1}$. 
  From \e{4.3} we obtain $ \nabla_{e_{i}}e_{5}=\nu e_{i}$, $ i=1,2.$ Thus  $D^{1}_{e_{1}}e_{5}=D^{1}_{e_{2}}e_{5}=0,$ where $D^{1}$ is the normal connection of $\mathcal L_{1}$ in $M$.
  It follows from  Gauss' equation and Lemma \ref{L:4.1} that the curvature tensor $R$ satisfies
 \begin{align}\label{4.8} \<R(e_{1},e_{2})e_{1},e_{j}\>=0,\;\; j=3,4,5.\end{align}
 Thus  \e{4.8} and Lemma \ref{L:4.2} imply that
$0 \equiv R(e_{1},e_{2})e_{1} \equiv (e_{2}\nu) e_{5} \, {\rm (mod}\; T_{1})$.
Hence  $e_{2}\nu=0$.
Similarly, by considering $R(e_{2},e_{1})e_{2}$, we also have $e_{1}\alpha=0$. After combining these with $D^{1}e_{5}=0$, we conclude that $\mathcal L_{1}$ has parallel mean curvature vector in $M$. Hence $T_{1}$ is a spherical distribution. 
Similarly,  $T_{2}$ is also a spherical distribution. Consequently, we obtain statement (c). 
\end{proof}

\begin{lem}\label{L:4.4} Under the hypothesis of Lemma \ref{L:4.2},  $M$ is locally a warped product $I\times_{\rho_{1}(t)} M_{1}^{2}\times_{\rho_{2}(t)} M_{2}^{2}$, where $t$ is function such that $e_{5}=\partial_{t}$ $($i.e., $e_{5}=\frac{\partial}{\partial t})$, $\rho_{1}$ and $\rho_{2}$ are two positive functions in $t$ and
$M_{1}^{2}, M_{2}^{2}$ are Riemannian 2-manifolds. 
\end{lem}
\begin{proof} This lemma follows from Lemma \ref{L:4.3} and a result of Hiepko \cite{H} (see also \cite[Theorem 4.4, p. ~ 90]{book}).
\end{proof}

Lemma 3.3 and  \e{4.4} imply  that $\mu$ depends only on $t$. Thus $\mu=\mu(t)$.

\begin{lem}\label{L:4.5}  Let $M$ be a Lagrangian submanifold of $\tilde M^5(4c)$ whose second fundamental form satisfies  \e{4.1} with $a,b,\mu \ne 0$. Then  we have $c=-\nu^{2}-\mu^{2}<0$. Thus $\mu$ satisfies
$\mu'(t){}^{2}=-4\mu^{2}(t)(c+\mu^{2}(t))$. \end{lem}
\begin{proof} Under the hypothesis, it follows from Gauss' equation and Lemma \ref{L:4.1}  that
$ \<R(e_{1},e_{3})e_{3},e_{1}\>=c+\mu^{2}.$
On the other hand, the definition of curvature tensor and  Lemma \ref{L:4.2}  imply that
$\<R(e_{1},e_{3})e_{3},e_{1}\>=-\nu^{2}$. Thus 
$c=-\nu^{2}-\mu^{2}<0$. By combining this with the definition of $\nu$, we obtain the lemma.
\end{proof}

 \section{More lemmas}
 
Next, we consider the case $a,\mu\ne 0$ and $b=0$. 

\begin{lem}\label{L:5.1} Let $M$ be a Lagrangian submanifold of $\tilde M^5(4c)$ whose second fundamental form satisfies  \e{4.1} with $a,\mu \ne 0$ and $b=0$. Then we have 
\begin{equation}\begin{aligned}\label{5.1}  
&\nabla_{e_{1}}e_{1}=\frac{e_{2}a}{3a} e_{2}+\frac{e_{3}a}{a}e_{3}+\frac{e_{4}a}{3a} e_{4}-\nu e_{5},
\\& \nabla_{e_{1}}e_{2}=-\frac{e_{2}a}{3a}  e_{1} -3 \o^{2}_{1}(e_{3})e_{3}-3 \o^{2}_{1}(e_{4})e_{4},
\\& \nabla_{e_{1}}e_{3}=-\frac{e_{3}a}{a}  e_{1} +3 \o^{2}_{1}(e_{3})e_{2}+ \o^{4}_{3}(e_{1})e_{4},
\\& \nabla_{e_{1}}e_{4}=-\frac{e_{4}a}{a}  e_{1} +3 \o^{2}_{1}(e_{4})e_{2}- \o^{4}_{3}(e_{1})e_{3},
\\& \nabla_{e_{2}}e_{1}=-\frac{e_{1}a}{3a}  e_{2} +3 \o^{2}_{1}(e_{3})e_{3}+ \o^{4}_{1}(e_{2})e_{4},
\\&\nabla_{e_{2}}e_{2}=\frac{e_{1}a}{3a} e_{1}+\frac{e_{3}a}{a}e_{3}+\frac{e_{4}a}{a} e_{4}-\nu e_{5},
\\& \nabla_{e_{2}}e_{3}=-3 \o^{2}_{1}(e_{3})e_{1}-\frac{e_{3}a}{a}  e_{2} + \o^{4}_{3}(e_{2})e_{4},\\
& \nabla_{e_{2}}e_{4}=- \o^{4}_{1}(e_{2})e_{1}-\frac{e_{4}a}{a}  e_{2} - \o^{4}_{3}(e_{2})e_{3},
\\& \nabla_{e_{3}}e_{1}= \o^{2}_{1}(e_{3})e_{2},\;\; \nabla_{e_{3}}e_{2}=-\o^{2}_{1}(e_{3})e_{1},
\\& \nabla_{e_{3}}e_{3}= \o^{4}_{3}(e_{3})e_{4}-\nu   e_{5},
\; \nabla_{e_{3}}e_{4}=- \o^{4}_{3}(e_{3})e_{3},\\& \nabla_{e_{4}}e_{1}= \o^{2}_{1}(e_{4})e_{2},\; \nabla_{e_{4}}e_{2}=-\o^{2}_{1}(e_{4})e_{1},
\; \\& \nabla_{e_{4}}e_{3}=  \o^{4}_{3}(e_{4})e_{4},\;  \nabla_{e_{4}}e_{4}=-\o^{4}_{3}(e_{4})e_{3}-\nu  e_{5},
\; \\&\nabla_{e_{5}}e_{3}=\o^{4}_{3}(e_{5}) e_{4},\;   \nabla_{e_{5}}e_{4}=-\o^{4}_{3}(e_{5})e_{5},
\; \\&  \nabla_{e_{i}}e_{5}=\nu e_{i},\;\; i\in \Delta,\;  
\nabla_{e_{k}}e_{j}=0,\;\; otherwise.\end{aligned}\end{equation}
with $\nu=\frac{1}{2}e_{5}(\ln \mu)=-e_{5}(\ln a).$ Moreover, we have 
\begin{align}\label{5.2}e_{j}\mu=0,\;\; j\in \Delta=\{1,2,3,4\}.\end{align}
\end{lem}
\begin{proof} Follows from Codazzi's equations via Lemma \ref{4.1} and \e{4.2}.
\end{proof}

\begin{lem}\label{L:5.2} Under the hypothesis of Lemma \ref{L:5.1}, we have

\vskip .04in
\begin{itemize}
\item[{\rm (i)}] $T_{0}$ is a totally geodesic distribution;

\item[{\rm (ii)}]  $T_{3}$ is a spherical  distribution, 
\end{itemize}
 where
$T_{0}={\rm Span}\{e_{5}\}$ and $T_{3}={\rm Span}\{e_{1},e_{2},e_{3},e_{4}\}.$
\end{lem}
\begin{proof} Clearly, $T_{0}$ is integrable. Moreover, since $\nabla_{e_{5}}e_{5}=0$ by Lemma \ref{L:5.1}, integral curves of $e_{5}$ are geodesics in $M^{5}$. Thus  statement (i) follows.
To prove statement (ii), we observe that the  integrability of $T_{3}$ follows  from \e{5.1}. Also,  \e{5.1} implies that the second fundamental form $\hat h$ of a leaf ${\mathcal L}$ of $T_{3}$ in $M^{5}$ is given by
$ \hat h(X,Y)=-\nu \hat g(X,Y)e_{5}$ for $ X,Y\in T\mathcal L,$ where $\hat g$ is the metric of $\mathcal L$. 
 Since $[e_{j},e_{5}]\mu =0$ by  \e{5.1} and $e_{j}\mu=0$,  for  $j\in \Delta$, we find 
$e_{i}e_{5}\mu-e_{5}e_{i}\mu=2 e_{1}\nu=0.$ Therefore $T_{3}$ is a spherical distribution.
\end{proof}

\begin{lem}\label{L:5.3} Under the hypothesis of Lemma \ref{L:5.1},  $M$ is locally a warped product $I\times_{\rho(t)} N^{4}$, where $t$ is function such that $e_{5}=\frac{\partial}{\partial t}$ and $\rho$ is a positive function in $t$  and
$N^{4}$ is a Riemannian 4-manifold. 
\end{lem}
\begin{proof} Follows from Lemma \ref{L:5.2} and Hiepko's theorem.\end{proof}

It follows from \e{5.2} and the definition of $\nu$ that $\mu=\mu(t)$ and $\nu=\nu(t)$.

\begin{lem}\label{L:5.4} Under the hypothesis of Lemma \ref{L:5.1}, we have
 \begin{align} \label{5.3} &\frac{d\nu}{dt}=-3\mu^{2}-\nu^{2}-c,\;\; \frac{d\mu}{dt}=2\mu\nu.\end{align}
\end{lem}
\begin{proof}  From Gauss' equation and \e{5.1} we find $\<R(e_{1},e_{5})e_{5},e_{1}\>=3\mu^{2}+c$. On the other hand,  \e{5.1} of Lemma \ref{L:5.1} yields $\<R(e_{1},e_{5})e_{5},e_{1}\>=-e_{5}\nu-\nu^{2}$. Thus we find the first equation of \e{5.3}. The second one follows immediately from  the definition of $\nu$ given in Lemma \ref{L:5.1}.
\end{proof}

\section{Improved $\delta(2,2)$-ideal Lagrangian submanifolds of ${\bf C}^{5}$}

\begin{thm}\label{T:6.1} Let $M$ be an improved $\delta(2,2)$-ideal Lagrangian submanifold in ${\bf C}^{5}$. Then it is one of the following Lagrangian submanifolds:

\begin{itemize}
\item[{\rm (a)}]  a $\delta(2,2)$-ideal Lagrangian minimal submanifold;
 
 \item[{\rm (b)}] an $H$-umbilical Lagrangian submanifold of ratio 4; 
 
 \item[{\rm (c)}]  a Lagrangian submanifold defined by
 \begin{equation}\label{6.1} L(\mu, u_{2},\ldots,u_{n})=\frac{e^{\frac{4 }{3} \ii\tan^{-1} \! \sqrt{\mu^{3}/(c^{2}-\mu^{3})}}}{\sqrt{{c^{2}}{\mu^{-1}}-\mu^{2}}+ \ii\mu}\phi( u_{2},\ldots,u_{n}),  \end{equation}
where $c$ is a positive real number and $\phi( u_{2},\ldots,u_{n})$ is a horizontal lift of a non-totally geodesic $\delta(2)$-ideal  Lagrangian minimal  immersion  in $CP^4(4)$.
  \end{itemize} \end{thm}
\begin{proof}  Assume that $M$ is an improved $\delta(2,2)$-ideal Lagrangian submanifold in ${\bf C}^{5}$. Then there exists an orthonormal frame $\{e_{1},\ldots,e_{5}\}$ such that \e{4.1} holds. If $\mu=0$, then $M$ is a minimal $\delta(2,2)$-ideal Lagrangian submanifold. Thus, we obtain case (a). If $\mu\ne 0$ and $a=b=0$,  we obtain case (b). 

Now, let us assume $a,\mu\ne 0$. Then Lemma \ref{L:4.5} implies $b=0$. So, by Lemmas \ref{L:5.1} we have \e{5.1} and $e_{j}\mu=0$, $j\in \Delta$. Further, by Lemma  \ref{L:5.3}, $M$ is locally a warped product $I\times_{\rho(t)} N^{4}$ with $e_{5}=\partial_{t}$. Moreover,  \ref{L:4.1} shows that the second fundamental form satisfies
 \begin{equation}\begin{aligned}\label{6.2} &h(e_{1},e_{1})=a Je_{1}+\mu Je_{5},\; h(e_{1},e_{2})=-a Je_{2},\;\\& h(e_{2},e_{2})=-a Je_{1}+\mu Je_{5},\;\\& h(e_{3},e_{3})=h(e_{4},e_{4})=\mu Je_{5},\;\; 
\\& h(e_{i},e_{5})= \mu Je_{i},\; i\in \Delta,  \\&h(e_{5},e_{5})=4\mu Je_{5},
\;  h(e_{i},e_{j})=0,\;\; otherwise.\end{aligned}\end{equation}

From Lemma \ref{L:5.4} we have the following differential system:
\begin{align} \label{6.3} &\frac{d\nu}{dt}=-3\mu^{2}-\nu^{2},\;\;  \frac{d\mu}{dt}=2\mu\nu.\end{align}
Let $\varphi(t)$ be a function satisfying 
$\dfrac{d\varphi}{dt}=-4\mu.$
Consider the map \begin{align}\label{6.4} \phi=e^{\ii \varphi}e_{5}.\end{align}
Then $\<\phi,\phi\>=1$. 
It follows from $\nabla_{e_{5}}e_{5}=0$,  $\frac{d\varphi}{dt}=-4\mu$ and \e{6.2} that 
$\tilde\nabla_{e_{5}}\phi=0$,
where $\tilde \nabla$ is the Levi-Civita connection of ${\bf C}^{5}$. Thus $\phi$ is independent of $t$.

Let $L$ denote the Lagrangian immersion of $M$ in ${\bf C}^{5}$. Then  \e{6.4} yields
\begin{align}\label{6.5} e_{5}=L_{t}=e^{-\ii \varphi}\phi(u_{1},\ldots,u_{4}),\end{align}
where $u_{1},\ldots,u_{4}$ are local coordinates of $N^{4}$. 
For each $j\in \Delta$, we obtain from $\nabla_{e_{j}}e_{5}=\nu e_{j}$ of Lemma \ref{L:5.1} and the first equation of \e{6.3} that
\begin{align}\label{6.6} \phi_{*}(e_{j})=\tilde\nabla_{e_{j}}\phi= e^{\ii \varphi}\tilde\nabla_{e_{j}}e_{5}=e^{\ii\varphi}(\nu+\ii \mu)e_{j}.\end{align} Thus 
\begin{equation}\label{6.7} \tilde\nabla_{e_{j}}(\phi_{*}(e_{i}))=e^{\ii\varphi}(\nu+\ii \mu)\tilde\nabla_{e_{j}}e_{i}.\end{equation}

In view of $\nabla_{e_{j}}e_{5}=\nu e_{j}$ and \e{6.2}, we may put
\begin{align}\label{6.8} \tilde\nabla_{e_{i}}e_{j}=\Big(\sum_{k=1}^{4}\Gamma^{k}_{ij}+\ii h^{k}_{ij} \Big)e_{k}-(\nu -\ii \mu)\delta_{ij} e_{5},\; \; i,j\in \Delta,
\end{align} for some functions $\Gamma^{k}_{ij}$.
Now, it follows  from \e{6.4}, \e{6.6}, \e{6.7}, and \e{6.8} that
\begin{equation}\begin{aligned}\label{6.9} \tilde\nabla_{e_{j}}(\phi_{*}(e_{i}))&=\sum_{\gamma=2}^{n} \(\Gamma^{k}_{ij}+\ii h^{k}_{ij}\)\phi_{*}(e_{k})- \(\mu^{2}+\nu^{2}\) \delta_{ij}\phi
\\&=\sum_{\gamma=2}^{n} \(\Gamma^{k}_{ij}+\ii h^{k}_{ij}\)\phi_{*}(e_{k})- \<\phi_{*}(e_{i}),\phi_{*}(e_{j})\>\phi.\end{aligned}\end{equation}

Since $M$ is a Lagrangian submanifold in ${\bf C}^5$, \e{6.4} and \e{6.6} show that $\ii \phi$ is perpendicular to each tangent space of $M$. Hence  $\phi$ is a horizontal immersion in the unit hypersphere $S^{9}(1)\subset {\bf C}^{5}$. Moreover, it follows from \e{6.9} that the second fundamental form of $\phi$ is the original second fundamental form of $M$ respect to to the second factor $N^{4}$ of the warped product $I\times_{\rho(t)}N^{4}$. Hence, $\phi$ is a minimal horizontal immersion in $S^{9}(1)$. Therefore, $\phi$ is a horizontal lift of a minimal Lagrangian immersion in $CP^{4}(4)$. Now, it follows from \e{6.2} that $\phi$ is a horizontal lift of a $\delta(2)$-ideal minimal Lagrangian submanifold of $CP^{4}(4)$. 

By direct computation we find
\begin{equation}\label{6.10}\tilde\nabla_{e_{\a}}\! \(L-\frac{e_{5} }{\nu+ \ii\mu} \)= 0,\;\; \a=1,\ldots,5.\end{equation} 
Thus, by \e{6.4}, up to translations the Lagrangian immersion $L$ is 
\begin{equation}\begin{aligned}\label{6.11} L=\frac{e^{-\ii \varphi}}{\nu+ \ii\mu}\phi(u_{1},\ldots,u_{4}) , \end{aligned}\end{equation}
where $\phi$ is a horizontal minimal immersion in $S^{9}(1)$ and $\nu,\varphi,\mu$ satisfy 
\begin{align}\label{6.12}\frac{d\nu}{dt}=-3\mu^{2}-\nu^{2},\;\;\frac{d\varphi}{dt}=-4\mu,\;\;   \frac{d\mu}{dt}=2\mu\nu.\end{align}
From \e{6.12} we find
\begin{align}\label{6.13} \frac{d\nu}{d\mu}+\frac{\nu}{2\mu}=-\frac{3\mu}{2\nu}.\end{align}
After solving \e{6.13} we get $\nu=\pm \sqrt{{c^{2}}{\mu^{-1}}-\mu^{2}}$
for some real number $c>0$.  Replacing $e_{5}$ by $-e_{5}$ if necessary, we have
\begin{align}\label{6.14}\nu= \sqrt{{c^{2}}{\mu^{-1}}-\mu^{2}}.\end{align}
It follows from \e{6.12} an \e{6.14} that 
$\varphi'(\mu)={-2}/{\sqrt{{c^{2}}{\mu^{-1}}-\mu^{2}}}$. By solving the last equation we find
$\varphi=-\frac{4 }{3} \ii\tan^{-1} \! \sqrt{\mu^{3}/(c^{2}-\mu^{3}})+c_{0}$ for some constant $c_{0}$. Therefore, we have the theorem after applying a suitable translation in $\mu$.
 \end{proof}

\begin{rem} Minimal $\delta(2,2)$-ideal Lagrangian submanifolds in complex space forms ${\bf C}^{5}$, $CP^{5}$ and $CH^{5}$ are classified in \cite{CP}. Also $\delta(2)$-ideal minimal Lagrangian submanifolds in $CP^{4}$ and $CH^{4}$ have been classified recently in \cite{sdsv}.
\end{rem}

Let $\gamma(t)$ be a unit speed curve in ${\bf C}^{*}$. We put \begin{align}\label{6.15}\gamma(t)=r(t)e^{i\theta(t)},\;\; \gamma'(t)=e^{i\zeta(t)}.\end{align}

The following result gives $H$-umbilical  submanifolds of ${\bf C}^5$ with ratio 4.

\begin{prop} \label{P:6.3} If $M$ is an $H$-umbilical Lagrangian submanifold of ${\bf C}^5$ of ratio 4, then $M$ is an open part of a complex extensor $\gamma\otimes \iota$ of the unit hypersphere $\iota: S^{4}(1)\subset \Bbb E^5$ via a generating curve $\gamma:I\to {\bf C}^{*}$ whose curvature satisfies $\kappa=4\theta'$.
\end{prop}
\begin{proof}  If $M$ is an $H$-umbilical Lagrangian submanifold of ${\bf C}^5$ with ratio 4, then the second fundamental form satisfies
\begin{equation} \begin{aligned} \notag&
 h(e_j,e_j)\hskip-.01in =\hskip-.01in \mu
J e_5, \;\; h(e_j, e_5)= \mu J e_j,\; \; j\in \Delta,
\\& h(e_5, e_5)= 4\mu J e_5,\; h(e_j,e_k)= 0,\;\;  1\leq j\ne k\leq 4, \end{aligned}\end{equation}  
for a nonzero function $\mu$. Thus Gauss' equation yields $K(e_1\wedge e_5
)=3\mu^2$. Hence $M$ is non-flat. Therefore, according to Theorem F, $M$ is an open part of a complex extensor of $\iota: S^{n-1}(1)\subset {\mathbb E}^{n}$ via a generating curve $\gamma:I\to {\bf C}^{*}$. 
It follows from \cite{tohoku} that the  functions $\varphi$ and $\mu$ in \e{4.1} are related with the two angle functions $\zeta$ and $\theta$ by $\varphi=\zeta'(t)=\kappa$ and $
\mu=\theta'(t)$. Thus whenever $\gamma$ is a unit speed curve satisfying $\kappa=4\theta'$,  the complex extensor $\gamma \otimes\iota\, $ is an $H$-umbilical Lagrangian submanifold of ratio 4. Conversely, every $H$-umbilical Lagrangian submanifold of ratio 4 in ${\bf C}^{n}$ can be obtained in such way.
\end{proof}

\section{Improved $\delta(2,2)$-ideal Lagrangian submanifolds of $CP^{5}$}

\begin{thm}\label{T:7.1} Let $M$ be an improved $\delta(2,2)$-ideal Lagrangian submanifold in $CP^5(4)$. Then it is one of the following Lagrangian submanifolds:

\begin{itemize}
\item[{\rm (1)}]  a $\delta(2,2)$-ideal Lagrangian minimal submanifold;
 
 \item[{\rm (2)}] an $H$-umbilical Lagrangian submanifold of ratio 4; 
 \item[{\rm (3)}]  a Lagrangian submanifold defined by 
 \begin{align} \label{7.1}&\hskip.1in L(\mu,u_{2},\ldots,u_{4})=\frac{1}{c}\Big(\sqrt{\mu} e^{\ii \theta}\phi,e^{3\ii \theta} (\sqrt{c^{2}-\mu^{3}-\mu}-\ii\mu^{\frac{3}{2}})\Big),\end{align}
 where $c$ is a positive real number, $\phi:N^4\to S^{9}(1)\subset {\bf C}^{5}$  is  a horizontal lift of a non-totally geodesic $\delta(2)$-ideal Lagrangian minimal immersion in $CP^4(4)$, and $\theta(\mu)$ satisfies
 \begin{align} \label{7.2} &\frac{d\theta}{d\mu}=\frac{1}{2\sqrt{c^{2}\mu^{-1}-\mu^{2}-1}}.\end{align}
  \end{itemize} \end{thm}
\begin{proof}   Under the hypothesis there is an orthonormal frame $\{e_{1},\ldots,e_{5}\}$ such that \e{4.1} holds. If $\mu=0$, then $M$ is a $\delta(2,2)$-ideal Lagrangian minimal submanifold. Thus we obtain case (1). 
If $\mu\ne 0$ and $a,b=0$, then $M$ is an $H$-umbilical Lagrangian submanifold of ratio 4, which gives case (2).

Next, assume that $a,\mu\ne 0$. Then Lemma \ref{L:4.5} implies $b=0$. So, by Lemmas \ref{L:5.1} we obtain \e{5.1} and \e{5.2}. Also, in this case $M$ is locally a warped product $I\times_{\rho(t)} N^{4}$ with $e_{5}=\partial_{t}$ according to Lemma  \ref{L:5.3}. 
From Lemma  \ref{L:4.1}, we find
 \begin{equation}\begin{aligned}\label{7.3} &h(e_{1},e_{1})=a Je_{1}+\mu Je_{5},\; h(e_{1},e_{2})=-a Je_{2},\;\\& h(e_{2},e_{2})=-a Je_{1}+\mu Je_{5},\\&h(e_{3},e_{3})=h(e_{4},e_{4})=\mu Je_{5},\; h(e_{5},e_{5})=4\mu Je_{5},
\\& h(e_{i},e_{5})= \mu Je_{i},\; i\in \Delta,\;
\;  h(e_{i},e_{j})=0,\; otherwise.\end{aligned}\end{equation}
By Lemma \ref{L:5.4} we have the following ODE system:
\begin{align} \label{7.4} &\frac{d\nu}{dt}=-1-\nu^{2}-3\mu^{2},\;\;\;  \frac{d\mu}{dt}=2\mu\nu.\end{align}
Let $\theta(t)$ be a function on $M$ satisfying 
\begin{align}\label{7.5} \theta'(t)={\mu}.\end{align}

Let  $L$ denote the  horizontal lift in $S^{11}(1)\subset {\bf C}^{6}$ of the Lagrangian immersion of $M$ in $CP^5(4)$ via  Hopf 's fibration. 
Consider the maps:
\begin{align}\label{7.6} \xi=\frac{e^{-3\ii \theta}\(e_{5}- \(\nu+{\ii\mu}\)L\)}{\sqrt{1+{\mu^2}+\nu^2}},\;\; \phi=\frac{e^{-\ii \theta}\(L+ \(\nu-\ii\mu\)e_5\)}{\sqrt{1+\mu^2+\nu^2}}.\end{align}
Then $\<\xi,\xi\>=\<\phi,\phi\>=1$. From $\nabla_{e_{j}}e_{5}=\nu e_{j}$, $j\in \Delta,$ and \e{7.4}, we find $\tilde \nabla_{e_j}\xi=0$. 
Moreover, it follows from Lemma \ref{L:5.1} and \e{7.3} that $\tilde \nabla_{e_{5}}e_{5}=4\ii \mu e_{5}-L$. Thus we also hhve$\tilde \nabla_{e_5}\xi=0$. Hence $\xi$ is a constant unit vector in ${\bf C}^{6}$. 
Similarly, we also have $\tilde \nabla_{e_5}\phi=0$. So $\phi$ is independent of $t$.
 Therefore, by combining \e{7.6} we find
\begin{align}\label{7.8}L=\frac{e^{\ii \theta}\phi-  e^{3\ii \theta}(\nu-\ii\mu)\xi}{\sqrt{1+\mu^2+\nu^2}}.\end{align}
Since $\phi$ is orthogonal to $\xi,\ii\xi$,   after choosing $\xi=(0,\ldots,0,1)\in {\bf C}^{6}$ we obtain 
\begin{align}\label{7.9} L=\frac{1}{\sqrt{1+\mu^2+\nu^2}}\(e^{\ii \theta}\phi, e^{3\ii \theta}(\nu-\ii\mu)\)\end{align}

It follows from \e{7.4} and \e{7.5} that 
\begin{align} \label{7.10} &\frac{d\nu}{d\mu}=-\frac{1+\nu^{2}+3\mu^{2}}{2\mu\nu},\;\;\;  \frac{d\theta}{d\mu}=\frac{1}{2\nu}.\end{align}
Solving the first differential equation in \e{7.10} gives
\begin{align} \label{7.11} &\nu=\pm \sqrt{c^{2}\mu^{-1}-\mu^{2}-1},\;\; c\in {\bf R}^{+}. \end{align}
 By replacing $e_{5}$ by $-e_{5}$ if necessary, we have $\nu=\sqrt{c^{2}\mu^{-1}-\mu^{2}-1}$.
Consequently, 
\begin{align} \label{7.12}L=\frac{1}{c} \(\sqrt{\mu} e^{\ii \theta}\phi, e^{3\ii \theta} (\sqrt{c^{2}-\mu^{3}-\mu}-\ii\mu^{\frac{3}{2}})\),\end{align}

It follows from \e{5.1}, \e{7.3} and the second formula in \e{7.6} that 
\begin{align}\label{7.13}\hat\nabla_{e_j}\phi=\frac{ce^{-\ii \theta}}{\sqrt{\mu}}e_j,\;\; j\in \Delta. \end{align}
Thus after applying \e{6.11} and \e{7.13} we derive that 
\begin{equation}\label{7.14} \hat\nabla_{e_{\b}}\hat\nabla_{e_\a}\phi=\sum_{\gamma=2}^{n} \(\Gamma^{k}_{ij}+\ii h^{k}_{ij}\)\phi_{*}(e_{k})- \<\phi_{*}(e_{i}),\phi_{*}(e_{j})\>\phi, \; i,j\in \Delta.\end{equation}
 Hence  $\phi$ is a horizontal immersion in $S^{9}(1)$. Moreover, it follows from \e{7.14} that the second fundamental form of $\phi$ is a scalar multiple of the original second fundamental form of $M$ restricted to the second factor of the warped product $I\times_{\rho}N$. Consequently, $\phi$ is a minimal horizontal immersion in $S^{{9}}(1)$ of a non-totally geodesic $\delta(2)$-ideal Lagrangian minimal submanifold of $CP^{4}(4)$. 

The converse is easy to verify. \end{proof}

\section{Improved $\delta(2,2)$-ideal Lagrangian submanifolds of $CH^{5}$}

\begin{thm}\label{T:8.1} Let $M$ be an improved $\delta(2,2)$-ideal Lagrangian submanifold in $CH^5(-4)$. Then $M$ is one of the following Lagrangian submanifolds:

\begin{itemize}
\item[{\rm (i)}]  a $\delta(2,2)$-ideal Lagrangian minimal  submanifold;

 \item[{\rm (ii)}] an $H$-umbilical Lagrangian submanifold of ratio 4;

 \item[{\rm (iii)}] a Lagrangian submanifold defined by 
\begin{align} \label{8.1}&\hskip.15in L(\mu,u_{1},\ldots,u_{4})= \frac{1}{c}\(\! \sqrt{\mu} e^{\ii \theta}\phi(u_{2},\ldots,u_{4}),e^{-\ii \theta} (\sqrt{\mu\hskip-.01in - \hskip-.01in \mu^{3} \hskip-.01in - \hskip-.01in c^{2}}-\ii\mu^{\frac{3}{2}})\)\! ,\end{align}
  where $c$ is a positive number,  $\phi:N^4\to H_{1}^{9}(-1)\subset {\bf C}^{5}_{1}$ 
  is  a horizontal  lift of a non-totally geodesic $\delta(2)$-ideal minimal Lagrangian immersion in $CH^4(-4)$, and $\theta(t)$ satisfies
  $\frac{d\theta}{d\mu}=\frac{1}{2}\sqrt{1-\mu^{2}-c^{2}\mu^{-1}};$
  
 \item[{\rm (iv)}]  a Lagrangian submanifold defined by 
\begin{align} \label{8.2}&\hskip.15in L(\mu,u_{1},\ldots,u_{4})= \frac{1}{c} \( e^{-\ii \theta} (\sqrt{\mu\hskip-.01in - \hskip-.01in \mu^{3} \hskip-.01in + \hskip-.01in c^{2}}-\ii\mu^{\frac{3}{2}}),  \sqrt{\mu} e^{\ii \theta}\phi(u_{2},\ldots,u_{4})\)\!,\end{align}
  where $c$ is a positive number,  $\phi:N^4\to S^{9}(1)\subset {\bf C}^{5}$ 
  is  a horizontal  lift of a non-totally geodesic $\delta(2)$-ideal minimal Lagrangian immersion in $CP^4(4)$, and $\theta(t)$ satisfies
  $\frac{d\theta}{d\mu}=\frac{1}{2}\sqrt{1-\mu^{2}+c^{2}\mu^{-1}};$
    
   \item[{\rm (v)}] a Lagrangian submanifold defined by 
   \begin{equation}\begin{aligned}\label{8.3} &\hskip.1in L(t,u_{1},\ldots,u_{4})=\frac{1}{\cosh t\hskip-.01in -\hskip-.01in \ii \sinh t}\! \(\! 2t \hskip-.01in +\hskip-.01in w\hskip-.01in+ \hskip-.01in\ii \! \(\! \cosh 2t \hskip-.01in -\hskip-.02in   \<\psi,\psi\> \hskip-.01in -\hskip-.01in \frac{1}{4}\)\! ,\right.
\\&\hskip1.4in \left. \psi ,\, 2t+w+\ii \! \( \! \cosh 2t \hskip-.01in -\hskip-.02in   \<\psi,\psi\>\hskip-.01in +\hskip-.01in \frac{1}{4}\)\!\), \end{aligned}\end{equation}
where $\psi(u_{1},\ldots,u_{4})$  is a non-totally geodesic $\delta(2)$-ideal  Lagrangian minimal immersion in ${\bf C}^{4}$ and up to a constant 
$w(u_{1},\ldots,u_{4})$ is the unique solution of the PDE system: $w_{u_{j}}\! =2\< \psi_{u_{j}},\ii  \psi\>\! ,\, j=1,2,3,4;$

    \item[{\rm (vi)}] a Lagrangian submanifold defined by 
   \begin{equation}\begin{aligned}\label{8.4} &\hskip.1in L(t,u_{1},\ldots,u_{4})=\frac{1}{\cosh t\hskip-.01in -\hskip-.01in \ii \sinh t}\! \(\! 2t \hskip-.01in +\hskip-.01in w+\ii \! \(\! \cosh 2t \hskip-.01in -\hskip-.02in   \<\psi,\psi\> \hskip-.01in -\hskip-.01in \frac{1}{4}\)\! ,\right.
\\&\hskip1.4in \left. \psi_{1},\psi_{2} ,\, 2t+\hskip-.01in w+\hskip-.01in\ii \! \( \! \cosh 2t \hskip-.01in -\hskip-.02in   \<\psi,\psi\>\hskip-.01in +\hskip-.01in \frac{1}{4}\)\!\), \end{aligned}\end{equation}
where $\psi=(\psi_{1},\psi_{2})$ is the direct product immersion of two non-totally geodesic Lagrangian  minimal immersions $\psi_{\alpha}: N_{\alpha}^{2}\to {\bf C}^{2},\, \alpha=1,2$, and up to a constant 
$w(u_{1},\ldots,u_{4})$ is the unique solution of the PDE system: $w_{u_{j}}\! =2\< \psi_{u_{j}},\ii  \psi\>\! ,\, j=1,2,3,4$.
  
   \end{itemize}
   \end{thm}
\begin{proof}    Under the hypothesis there exists an orthonormal frame $\{e_{1},\ldots,e_{5}\}$ such that \e{4.1} holds. 
\vskip.05in

{\it Case} (1) $\mu=0$. In this case, we obtain case (i) of the theorem. 
\vskip.05in

{\it Case} (2): $\mu\ne 0$ {\it and} $a,b=0$. In this case $M$ is an $H$-umbilical Lagrangian submanifold with ratio 4, which gives case (ii).

\vskip.05in

{\it Case} (3): $\mu\ne 0$ {\it and at least one of $a,b$ is nonzero}. Without loss of generality, we may assume $a\ne 0$ and $\mu>0$. We divide this into two cases.

\vskip.05in

{\it Case} (3.a): $a,\mu\ne 0$ {\it and}  $b=0$. By Lemmas \ref{L:5.1} we obtain \e{5.1} and \e{5.2}. Also,  $M$ is locally a warped product $I\times_{\rho(t)} N^{4}$ with $e_{5}=\partial_{t}$ according to Lemma  \ref{L:5.3}. 
From Lemma  \ref{L:4.1} we find
 \begin{equation}\begin{aligned}\label{8.5} &h(e_{1},e_{1})=a Je_{1}+\mu Je_{5},\; h(e_{1},e_{2})=-a Je_{2},\;\\& h(e_{2},e_{2})=-a Je_{1}+\mu Je_{5},\\&h(e_{3},e_{3})=h(e_{4},e_{4})=\mu Je_{5},\; h(e_{5},e_{5})=4\mu Je_{5},
\\& h(e_{i},e_{5})= \mu Je_{i},\; i\in \Delta,\;
\;  h(e_{i},e_{j})=0,\; otherwise.\end{aligned}\end{equation}

Let  $L$ be a horizontal immersion of $M$ in $H^{11}_{1}(-1)\subset {\bf C}_{1}^{6}$ of the Lagrangian immersion of $M$ in $CH^5(-4)$ via  Hopf 's fibration and $\theta(t)$ a function satisfying 
\begin{align}\label{8.6} \frac{d\theta}{dt}={\mu}.\end{align}
From Lemma \ref{L:5.4} we obtain the following ODE system:
\begin{align} \label{8.7} &\frac{d\nu}{dt}=1-3\mu^{2}-\nu^{2},\;\;\;  \frac{d\mu}{dt}=2\mu\nu.\end{align}
It follows from \e{8.6} and \e{8.7} that 
\begin{align} \label{8.8} &\frac{d\nu}{d\mu}=\frac{1-3\mu^{2}-\nu^{2}}{2\mu\nu},\;\;\;  \frac{d\theta}{d\mu}=\frac{1}{2\nu}.\end{align}
Solving the first differential equation in \e{8.8} gives
$\nu=\pm \sqrt{1- \mu^{2}-k\mu^{-1}}$ for some real number $k$.
By replacing $e_{5}$ by $-e_{5}$ if necessary, we find 
\begin{align} \label{8.9} &\nu=\sqrt{1- \mu^{2}-k\mu^{-1}},\;\;\;  \frac{d\theta}{d\mu}=\frac{1}{2\sqrt{1- \mu^{2}-k\mu^{-1}}}.\end{align}

It follows from \e{8.7} that 
$\frac{d}{dt}(1-\mu^2-\nu^{2})=-2\nu(1-\mu^2-\nu^{2}).$
Since this  equation for  $y(t)=1-\mu^2-\nu^{2}=k\mu^{{-1}}$ has a unique solution for each given initial condition, each solution either vanishes identically or is nowhere zero. 

\vskip.05in

{\it Case} (3.a.1): $\mu^2+\nu^{2}<1$. In this case,  \e{8.9} implies $k>0$. Thus we may put $k=c^{2},\, c>0$.
Consider the maps:
\begin{align}\label{8.10}  \eta=\frac{e^{-3 \ii \theta}(e_{5}- (\nu+\ii \mu)L
\big)}{\sqrt{1-\mu^2-\nu^{2}}},\;\; \phi=\frac{e^{-\ii \theta}\( \(\nu-\ii\mu\)e_5-L\)}{\sqrt{1-\mu^2-\nu^2}}.\end{align}
Then $\<\eta,\eta\>=1$ and $\<\phi,\phi\>=-1$. From $\nabla_{e_{j}}e_{5}=\nu e_{j}$, $j\in \Delta,$ and \e{8.5}, we obtain $\tilde \nabla_{e_j}\xi=0$, where $\tilde \nabla$ is the Levi-Civita connection of ${\bf C}^{6}_{1}$. 
 Lemma \ref{L:5.1} and \e{8.5} give $\tilde \nabla_{e_{5}}e_{5}=4\ii \mu e_{5}+L$. Thus we find $\tilde \nabla_{e_5}\xi=0$. So $\eta$ is a constant unit vector. 
Also, we find $\tilde \nabla_{e_5}\phi=0$. Hence $\phi$ is independent of $t$.
From \e{8.10}  we get
\begin{align}\label{8.12}L=- \frac{e^{\ii \theta}\phi +  e^{-\ii \theta}(\nu-\ii\mu)\eta}{\sqrt{1-\mu^2-\nu^2}}.\end{align}
Since $\phi$ is orthogonal to $\eta,\ii\eta$ and $\eta$ is a constant unit space-like vector,  we conclude from \e{8.9} and \e{8.12} that $L$ is congruent to  \e{8.1}.
Next, by applying the same method of the proof of Theorem \ref{T:7.1}, we conclude that
 $\phi$ is a horizontal immersion in  $H_{1}^{9}(-1)$ whose second fundamental form  is a scalar multiple of the original second fundamental form restricted to the second factor of  $I\times_{\rho}N$. Consequently, $\phi$ is a minimal horizontal immersion in $H^{9}_{1}(-1)$ of a non-totally geodesic $\delta(2)$-ideal Lagrangian minimal submanifold of $CH^{4}(-4)$. This gives case (iii).

 \vskip.05in

{\it Case} (3.a.2): $\mu^2+\nu^{2}>1$. In this case \e{8.8} implies $k<0$. Thus we may put $k=-c^{2},\, c>0$. Now, we consider the maps:
\begin{align}\label{8.13}  \eta=\frac{e^{-3 \ii \theta}(e_{5}- (\nu+\ii \mu)L
\big)}{\sqrt{\mu^2+\nu^{2}-1}},\; \; \phi=\frac{e^{-\ii \theta}\( \(\nu-\ii\mu\)e_5-L\)}{\sqrt{\mu^2+\nu^2-1}}\end{align}
instead. Then $\<\phi,\phi\>=-\<\eta,\eta\>=1$.  
By applying similar arguments as case (3.a.1), we know that $\eta$ is a constant  time-like vector and $\phi$ is independent of $t$ and orthogonal to $\eta,\ii \eta$. Moreover, we may prove that $\phi$ is a minimal Legendre immersion in $S^{9}(1)$. Therefore  we have case (iv) after choosing $\eta=(1,0,\ldots,0)$.

\vskip.05in

{\it Case} (3.a.3): $\mu^2+\nu^{2}=1$. In this case system \e{8.7} gives
$\frac{d\nu}{dt}=2(\nu^{2}-1)$ and $ \mu=\pm \sqrt{1-\nu^{2}}.$
Solving these and applying a suitable translations in $t$, we find
\begin{align}\label{8.15}  &\mu=\sech 2t
,\;\; \nu=-\tanh2t.\end{align} 

It follows from $\nabla_{e_5}e_5=0$, \e{8.5} and \e{8.15} that the horizontal lift $L$ of the Lagrangian immersion of $M$ in $CH^{5}(-4)\subset {\bf C}^{6}_{1}$ satisfies
\begin{align}\label{8.16}  &L_{tt}-4\ii (\sech2t)L_t-L=0.\end{align}
Solving this second order differential equation gives
\begin{equation}\begin{aligned}\label{8.17} L=&\frac{\phi(u_1,\ldots,u_4)+B(u_1,\ldots,u_4)(2t+\ii \cosh 2t)}{\cosh t-\ii \sinh t},\end{aligned}\end{equation}
where $\phi(u_1,\ldots,u_4)$ and $B(u_1,\ldots,u_4)$ are ${\bf C}^{6}_1$-valued functions.

On the other hand, it follows from Lemma \ref{L:5.1}, \e{8.5} and \e{8.15} that 
\begin{align}\label{8.18}  &L_{tu_j}=(\ii \sech2t-\tanh2t)L_{u_j}, \;\; j\in \Delta.\end{align}
Substituting \e{8.17} into \e{8.18} shows that $B$ is a constant vector $\zeta$.  Thus
\begin{equation}\label{8.19} L(t,u_1,\ldots,u_4)=\frac{\phi(u_1,\ldots,u_4)}{\cosh t-\ii \sinh t}+\frac{(2t+\ii \cosh 2t)}{\cosh t-\ii \sinh t}\zeta,\end{equation}

Since $\<L,L\>=-1$, \e{8.19} implies 
\begin{equation}\label{8.20}-\cosh 2t=\<\phi,\phi\>+\<\phi,(4t+2\ii  \cosh 2t) \zeta\>
+(4t^2+\cosh^2(2t))\<\zeta,\zeta\>.\end{equation}
Since $\phi_t=0$, by  differentiating \e{8.20} with respect $t$ we find 
\begin{equation}\begin{aligned}\label{8.21}-\sinh 2t =2t\<\phi,\zeta\>+2 \sinh 2t\<\phi,\ii\zeta\>
+(4t+\sinh 4t)\<\zeta,\zeta\>.\end{aligned}\end{equation}
We find from \e{8.21} at $t=0$  that
$\<\phi,\zeta\>=0.$ Thus \e{8.21} gives
\begin{equation}\begin{aligned}\label{8.22}0=\sinh 2t(1+ \<\phi,\ii \zeta\>)
+(4t+\sinh 4t)\<\zeta,\zeta\>.\end{aligned}\end{equation}
Differentiating \e{8.22} gives $\<\phi,\ii \zeta\>=-\frac{1}{2}-2 \<\zeta,\zeta\>.$ Thus \e{8.19} yields
$\<\phi ,\ii \zeta\>=-\tfrac{1}{2}$ and $  \<\zeta,\zeta\>=0 $. Now, we find from \e{8.20} that 
$\<\phi,\phi\>=0$. Consequently we have
\begin{align}\label{8.24} \<\phi,\phi\>=\<\zeta,\zeta\>=\<\phi,\zeta\>=0,\;\; \<\phi,\ii \zeta\>=-\tfrac{1}{2}.\end{align}
Since $\zeta$ is a constant light-like vector,  we may put
\begin{align}\label{8.25} &\zeta=(1,0,\ldots,0,1),\;\; \phi=\(a_1+\ii b_1,\ldots,a_{6}+\ii b_{6}\).\end{align}
It follows from \e{8.24} and  \e{8.25} that $a_{6}=a_1$ and $b_{6}=b_{1}+\tfrac{1}{2}$. Therefore
\begin{align}\label{8.27}  \phi=\(a_1+\ii b_1,a_2+\ii b_2,\ldots,a_{1}+\ii (b_{1}+\tfrac{1}{2})\).\end{align}
Now, by using $\<\phi,\phi\>=0$ and \e{8.27}, we find $\psi=\(a_{2}+\ii b_{2},\ldots,a_{5}+\ii b_{5}\)$ and
$b_1=-\tfrac{1}{4}-\<\psi,\psi\>$.
Combining  these with \e{8.27}  yields
\begin{align}\label{8.29} \phi=\(\! w- \ii \<\psi,\psi\>-\frac{\ii}{4},\psi ,w- \ii \<\psi,\psi\>+\frac{\ii}{4}\) \end{align}
with $w=a_{1}$. It follows from \e{8.25} and \e{8.29} that $\<\phi_{u_j}, \zeta\>=\<\phi_{u_j},\ii \zeta\>=0$.
Thus, by applying $\<L_{u_j},\ii L\>=0,\, j\in \Delta$, we find from \e{8.19} that  $\<\phi_{u_j},\ii \phi\>=0$.
 
On the other hand,  \e{8.29} implies that
\begin{align}\label{8.30}\<\phi_{u_j},\ii \phi\>=-\tfrac{1}{2}w_{u_{j}}+\<\psi_{u_{j}}, \ii\psi\>\end{align}
with $w_{u_{j}}=\frac{\partial w}{\partial u_j}$.
Therefore  $w$ satisfies the PDE system: $w_{u_{j}}=2\< \psi_{u_{j}},\ii  \psi\>.$

Now, we derive from \e{8.19}, \e{8.25} and \e{8.27} that 
\begin{equation}\begin{aligned}\label{8.31} &L=\frac{1}{\cosh t-\ii \sinh t}\(2t+w+\ii \(\cosh 2t-  \<\psi,\psi\>-\frac{1}{4}\),\right.
\\&\hskip1.4in \left. \psi ,2t+w+\ii \(\cosh 2t-  \<\psi,\psi\>+\frac{1}{4}\)\!\). \end{aligned}\end{equation}
It follows from \e{8.31} that
\begin{equation}\begin{aligned}\label{8.32} &L_{u_{j}}=\frac{1}{\cosh t-\ii \sinh t} \Big(w_{u_{j}}-\ii  \<\psi,\psi\>_{u_{j}} \! ,
 \, \psi_{u_{j}} ,w_{u_{j}}-\ii \<\psi,\psi\>_{u_{j}}\! \Big). \end{aligned}\end{equation}
Thus we find  $\<\psi_{u_{j}},\psi_{u_{k}}\>= \cosh 2t\<L_{u_{j}},L_{u_{k}}\>$
which implies that $\psi$ is an immersion  in ${\bf C}^{4}$. Also, we find from \e{8.32} and $ \<L_{u_{j}},\ii L_{u_{k}}\>=0$ that
$\<\psi_{u_{j}},\ii \psi_{u_{k}}\>=0$. Thus $\psi$ is a Lagrangian immersion. Now, by applying an argument  similar to  the last part of the proof of \cite[Theorem 6.1]{cdv}, we conclude that 
$$\psi_{u_{j}u_{k}}=\sum_{i=1}^{4}(\Gamma^{i}_{jk}+\ii h^{i}_{jk})\phi_{u_{i}},\;\; j,k\in \Delta.$$
Therefore, according to \e{8.5}, $\psi$ is a $\delta(2)$-ideal minimal Lagrangian immersion in ${\bf C}^{4}$. Consequently, we obtain case (v) of the theorem.

\vskip.05in

{\it Case} (3.b): $a,b,\mu\ne 0$. We obtain case (vi) of the theorem by applying the same argument as case (3.a.3).
\end{proof}

{\bf Acknowledgement.} The authors  thank the referee and Dr. Luc Vrancken for pointing out an error in the original version of this paper.


\begin{thebibliography}{12}


\bibitem{CC} I. Castro and B.-Y. Chen, Lagrangian surfaces in complex Euclidean plane via spherical and hyperbolic curves, {\it Tohoku Math. J.} {\bf 58} (2006), 565--579.

\bibitem{tohoku} B.-Y. Chen, Complex extensors and Lagrangian submanifolds in complex Euclidean spaces, {\it Tohoku Math. J.} {\bf 49} (1997), 277--297.

\bibitem{israel} B.-Y. Chen, {Interaction of Legendre curves and Lagrangian submanifolds,} {\it Israel J. Math}. {\bf 99} (1997), 69--108.

\bibitem{tohoku2} B.-Y. Chen, Representation of flat Lagrangian H-umbilical submanifolds in complex Euclidean spaces, {\it Tohoku Math. J.} {\bf 51} (1999),  13--20. 

\bibitem{c00a} B.-Y. Chen, Some new obstruction to minimal and Lagrangian isometric immersions, {\it Japan. J. Math.} {\bf 26} (2000), 105--127.

 \bibitem{c00b} B.-Y. Chen,  {Ideal Lagrangian immersions in complex space forms}, {\it Math. Proc. Cambridge Philos. Soc.} {\bf 128} (2000),  511--533.

\bibitem{book} B.-Y. Chen, Pseudo-Riemannian geometry, $\delta$-invariants and Applications, World Scientific, Hackensack, NJ, 2011.

\bibitem{CD}  B.-Y. Chen and F. Dillen, Optimal general inequalities for Lagrangian submanifolds in complex space forms, {\it J. Math. Anal. Appl.} {\bf 379} (2011),  229--239.

\bibitem{cdvv}  B.-Y. Chen, F. Dillen, J. Van der Veken and L. Vrancken, Curvature inequalities for Lagrangian submanifolds: the final solution, arXiv:1307.1497.

\bibitem{cdvv2}  B.-Y. Chen, F. Dillen, L. Verstraelen and L. Vrancken, {An exotic totally real minimal immersion of $S^3$ in $CP^3$ and its characterization},  {\it Proc. Roy. Soc. Edinburgh Sec. A Math.} {\bf 126} (1996), 153--165.

\bibitem{cdv}  B.-Y. Chen, F. Dillen and L. Vrancken,  Lagrangian submanifolds in complex space forms attaining equality in a basic inequality, {\it J. Math. Anal. Appl.} {\bf 386} (2012), 139--152.

\bibitem{CO} B.-Y. Chen and K. Ogiue, {On totally real submanifolds,} {\it Trans. Amer. Math. Soc.} {\bf 193} (1974), 257--266.

\bibitem{CP} B.-Y. Chen and A. Prieto-Mart\'{i}n, Classification of Lagrangian submanifolds in complex space forms  satisfying a basic equality involving $\delta(2,2)$, {\it Differential Geom. Appl.} {\bf 30} (2012), 107--123.

\bibitem{sdsv} F. Dillen, C. Scharlach, K. Schoels and L. Vrancken, {Special Lagrangian 4-folds with $SO(2)\rtimes S_{3}$-symmetry in complex space forms}, preprint.

\bibitem{H} S. Hiepko, {Eine innere Kennzeichung der verzerrten Produkte}, {\it Math. Ann}. {\bf 241} (1979), 209--215.

\bibitem{R} H. Reckziegel, Horizontal lifts of isometric immersions into the bundle space of a pseudo-Riemannian submersion, {\it Lecuture Notes in Math.} {\bf 1156} (1985), 264--279.


\end{thebibliography}
\end{document}